\documentclass[cmp]{svjour}

\usepackage[margin=2cm]{geometry}
\usepackage{amsmath,amssymb,array,booktabs,cite,dsfont,graphicx,mathtools,multirow,rotating,stmaryrd,tensor,verbatim}
\usepackage[bookmarksnumbered,linktocpage,pdfstartview=FitH]{hyperref}
\usepackage[all]{hypcap}
\usepackage[margin=2cm]{geometry}
\journalname{Communications in Mathematical Physics}

\graphicspath{{figures/}}

\newcolumntype{M}[1]{>{$}{#1}<{$}}
\newcolumntype{C}[1]{>{\centering}m{#1}}

\newcolumntype{B}[1]{>{\mathbf\bgroup}{#1}<{\egroup}}
\newcolumntype{K}{>{\lvert}{c}<{\rangle}}

\numberwithin{equation}{section}

\DeclareMathOperator{\Aut}{Aut}
\DeclareMathOperator{\Str}{Str}

\DeclareMathOperator{\Tr}{Tr} 
\DeclareMathOperator{\Det}{Det}

\DeclareMathOperator{\diag}{diag}
\DeclareMathOperator{\Iso}{Iso}
\DeclareMathOperator{\Hom}{Hom}
\DeclareMathOperator{\SO}{SO}

\DeclareMathOperator{\SL}{SL}
\DeclareMathOperator{\SU}{SU}
\DeclareMathOperator{\Sp}{Sp}

\newcommand{\pmtwo}[4]{\begin{pmatrix}#1 & #2 \\ #3 & #4 \end{pmatrix}}

\newcommand{\be}{\begin{equation}}
\newcommand{\ee}{\end{equation}}
\newcommand{\bea}{\begin{eqnarray}}
\newcommand{\eea}{\end{eqnarray}}

\newcommand{\half}{\tfrac{1}{2}}

\newcommand{\J}{\mathfrak{J}}

\newcommand{\JA}{\mathfrak{J}^{\mathds{A}}_{3}}
\newcommand{\JAs}{\mathfrak{J}^{\mathds{A}^s}_{3}}

\newcommand{\alg}{\mathds{A}}

\newcommand{\F}{\mathds{F}}
\newcommand{\R}{\mathds{R}}
\newcommand{\C}{\mathds{C}}
\newcommand{\Q}{\mathds{H}}
\newcommand{\Z}{\mathds{Z}}
\newcommand{\Oct}{\mathds{O}}

\newcommand{\FTS}{\mathfrak{F}}

\newcommand{\Fnt}{\mathfrak{F}^{2,n}}
\newcommand{\Fnf}{\mathfrak{F}^{6,n}}
\newcommand{\Rstr}[1]{\mathfrak{Str}_0(#1)}

\newcommand{\Jn}{\mathfrak{J}_{1,n-1}}
\newcommand{\Jnt}{\mathfrak{J}_{1,n-1}}
\newcommand{\Jnf}{\mathfrak{J}_{5,n-1}}

\begin{document}

























\title{{Explicit Orbit Classification of Reducible Jordan Algebras and Freudenthal Triple Systems}}
\titlerunning{Explicit Orbit Classification of Reducible Jordan Algebras and Freudenthal Triple Systems}

\author{L Borsten \inst{1}\and MJ Duff\inst{2} \and S Ferrara\inst{3}\fnmsep\inst{4}\fnmsep\inst{5}\and A Marrani\inst{3}\and W Rubens\inst{2}}
\institute{INFN Sezione di Torino \&  Dipartimento di Fisica Teorica,
Universit\`a di Torino \\
Via Pietro Giuria 1, 10125 Torino, Italy\\ \email{borsten@to.infn.it}
\and
Theoretical Physics, Blackett Laboratory, Imperial College London,\\
London SW7 2AZ, United Kingdom \\ \email{m.duff@imperial.ac.uk; w.rubens@imperial.ac.uk}
\and
Physics Department, Theory Unit, CERN,\\
CH -1211, Geneva 23, Switzerland \\ \email{sergio.ferrara@cern.ch; alessio.marrani@cern.ch}
\and
INFN - Laboratori Nazionali di Frascati,\\
Via Enrico Fermi 40, I-00044 Frascati, Italy
\and
Department of Physics and Astronomy,\\
University of California, Los Angeles, CA 90095-1547,USA}

\authorrunning{L Borsten \emph{et al}}

\date{\today}

\maketitle
\begin{abstract}
We determine explicit orbit representatives of \textit{reducible}
Jordan algebras and  their corresponding Freudenthal triple
systems. This work has direct application to the classification of
extremal black hole solutions of $\mathcal{N}=2$, $4$ locally
supersymmetric theories of gravity coupled to an arbitrary number of
Abelian vector multiplets in $D=4$, $5$ space-time dimensions.

\end{abstract}

\tableofcontents

\section{Introduction}

The present investigation is devoted to the study of the explicit
representatives of the orbits of \textit{reducible} cubic Jordan
algebras, and of their corresponding Freudenthal triple systems (FTS). 
This is in the spirit of previous analyses by Shukuzawa
\cite{Shukuzawa:2006}, which in turn was inspired \textit{e.g.} by
Jacobson \cite{Jacobson:1961} and Krutelevich \cite
{Krutelevich:2002,Krutelevich:2003ths,Krutelevich:2004}. By reducible we mean here that the cubic norm of the underlying Jordan algebra is a factorisable homogeneous polynomial of degree 3, as opposed to the  irreducible, \emph{i.e. non-factorisable}, cases treated in the previous works \cite{Shukuzawa:2006,Krutelevich:2004}. In a
companion paper \cite{ICL-2}, the results of the present analysis
have been used to classify extremal black hole solutions in locally
supersymmetric theories of gravity with $\mathcal{N}=2$ or $4$
supercharges in $D=4$ and $5$ space-time dimensions, coupled to an
arbitrary number of (Abelian) vector multiplets. This paper aims at
completing and refining previous investigations \cite
{Ferrara:1997uz,Ferrara:2006xx,Bellucci:2006xz,Ferrara:2007tu,Cerchiai:2009pi,Ceresole:2010nm,Cerchiai:2010xv,Andrianopoli:2010bj,Borsten:2010aa}%
, and it also provides an alternative approach with respect to the
analysis based on nilpotent orbits of symmetry groups characterising
the $D=3$ time-like reduced gravity theories \cite
{Breitenlohner:1987dg,Gunaydin:2000xr,Gunaydin:2005zz,Gunaydin:2005gd,Pioline:2005vi,Gunaydin:2005mx,Gunaydin:2007bg,Gunaydin:2007qq, Bossard:2009at,Bossard:2009mz,Fre:2011aa}.\medskip

The paper is organised as follows. In \autoref{sec:J} the cubic Jordan algebras construction related to $D=5$ supergravity \cite{Gunaydin:1983bi,Gunaydin:1983rk,Gunaydin:1984ak,Ferrara:2006xx} is summarised. After some
introductory background in  \autoref{sec:J-constr}  the orbits of the  \textit{reducible}
Lorentzian spin factors (then generalised to an arbitrary
pseudo-Euclidean signature) are derived in Sec. \autoref{sec:redJ}. The peculiar cases of the so-called $\mathcal{N}=2$ $%
STU$, $ST^{2}$ and $T^{3}$ supergravity models in $D=5$ are
considered in  \autoref{sec:d5special}.

Then,  \autoref{sec:FTS} studies the Freudenthal triple systems.
These are defined both axiomatically and in relation to possibly
underlying cubic
Jordan algebras, respectively in \autoref{sec:FTS-axio} and \autoref{sec:FTS-J}. Their automorphism group is recalled in \autoref{sec:FTS-auto}.  This is followed in 
\autoref{sec:Spin-Factors} by our main result, presented in \autoref{thm:redcanforms}, classifying the orbit
representatives of Freudenthal systems associated to
\textit{reducible} spin factors, including both the ``large'' rank $4$ orbits and the ``small'' orbits which split into three sub-classes, ranging
from rank $3$ to $1$. Physically, the small orbits correspond to black holes which classically exhibit naked singularities. As for the Jordan algebra analysis worked out in  \autoref{sec:J}, the $STU$, $ST^{2}$ and $T^{3}$ models
require a separate treatment, which is given in the concluding \autoref {sec:d4special}. Our main results in this case are presented in \autoref{thm:st2canform}  and \autoref{thm:T4}. In particular, it is shown that, since the $T^{3}$ Jordan algebra has only one non-trivial rank,  the rank 2 orbit of the FTS is not present in this case.

\section{Orbits of Reducible Cubic Jordan Algebras}\label{sec:J}

\subsection{Construction}\label{sec:J-constr}
A Jordan algebra $\mathfrak{J}$ is vector space defined over a
ground field $\mathds{F}$ (not of characteristic 2) equipped with a bilinear product
satisfying \cite{Jordan:1933a, Jordan:1933b, Jordan:1933vh}
\begin{equation}\label{eq:Jid}
A\circ B =B\circ A,\quad A^2\circ (A\circ B)=A\circ (A^2\circ B), \quad\forall\ A, B \in \mathfrak{J}.
\end{equation}
However, the 5-dimensional supergravities
\cite{Gunaydin:1984ak,Gunaydin:1983bi,Gunaydin:1983rk} are
characterised specifically by the class of \emph{cubic} Jordan
algebras, developed in \cite{Springer:1962,McCrimmon:1969}. We
sketch their construction here, following the presentation of
\cite{McCrimmon:2004}.

\begin{definition}[Cubic norm] A \emph{cubic norm} is a homogeneous map of degree three
\begin{equation}\label{eq:cubicnorm}
N:V\to \mathds{F}, \quad\text{s.t.} \quad N(\alpha A)=\alpha^3N(A), \quad \forall \alpha \in \mathds{F}, A\in V
\end{equation}
such that its linearization,
\begin{equation}
N(A, B, C):=\frac{1}{6}(N(A+ B+ C)-N(A+B)-N(A+ C)-N(B+
C)+N(A)+N(B)+N(C))
\end{equation}
is trilinear.
\end{definition}
Let $V$ be a vector space equipped with a cubic norm. If $V$ further
contains a base point $N(c)=1, c\in V$  one may define the following
four maps:
\begin{subequations}\label{eq:cubicdefs}
\begin{enumerate}
\item The trace,
    \begin{equation}
    \Tr(A)=3N(c, c, A),
    \end{equation}
\item A quadratic map,
    \begin{equation}
    S(A)=3N(A, A, c),
    \end{equation}
\item A bilinear map,
    \begin{equation}
    S(A, B)=6N(A, B, c),
    \end{equation}
\item A trace bilinear form,
    \begin{equation}\label{eq:tracebilinearform}
    \Tr(A, B)=\Tr(A)\Tr(B)-S(A, B).
    \end{equation}
\end{enumerate}
\end{subequations}
A cubic Jordan algebra $\mathfrak{J}$ with multiplicative identity $\mathds{1}=c$ may be derived from any such vector space if $N$ is \emph{Jordan cubic}.
\begin{definition}[Jordan cubic norm] A cubic norm is \emph{Jordan} if
\begin{enumerate}
\item The trace bilinear form \eqref{eq:tracebilinearform} is non-degenerate.
\item The quadratic adjoint map, $\sharp\colon\mathfrak{J}\to\mathfrak{J}$, uniquely defined by $\Tr(A^\sharp, B) = 3N(A, A, B)$, satisfies
    \begin{equation}\label{eq:Jcubic}
    (A^{\sharp})^\sharp=N(A)A, \qquad \forall A\in \mathfrak{J}.
    \end{equation}
\end{enumerate}
\end{definition}
The Jordan product is given by
\begin{equation}\label{eq:J3prod}
A\circ B := \half\big(A\times B+\Tr(A)B+\Tr(B)A-S(A,
B)\mathds{1}\big),
\end{equation}
where,
\begin{equation}\label{eq:FreuProduct}
A\times B := (A+B)^\sharp-A^\sharp-B^\sharp.
\end{equation}
Finally, the Jordan triple product is defined as
\begin{equation}\label{eq:Jtripleproduct}
\{A,B,C\}:=(A\circ B)\circ C + A\circ (B\circ C)-(A\circ C)\circ B.
\end{equation}
\begin{definition}[Irreducible idempotent] An element $E\in\mathfrak{J}$  is an \emph{irreducible idempotent} if
\be
E\circ E=E,\qquad \Tr(E)=1.
\ee
\end{definition}


There are many good references on the symmetries associated with
Jordan algebras and, in particular, on the exceptional Lie groups, see for example
\cite{Schafer:1966,Jacobson:1971,Ramond:1976aw,Springer:2000} and in
particular \cite{Krutelevich:2004, Shukuzawa:2006,Yokota:2009}. In the following we  restrict our attention to the case $\F=\R$.

  \begin{definition}[(Reduced) Structure group $\Str(\mathfrak{J})$] Invertible $\mathds{R}$-linear transformations $\tau$ preserving the cubic norm up to a fixed scalar factor,
   \be
\Str(\J):=\{\tau\in \Iso_{\R}(\J)| N(\tau A)=\lambda N(A), \lambda\in\R\}.
\ee
   The reduced structure group is the subgroup of invertible $\mathds{R}$-linear transformations $\tau$ preserving the cubic
 norm:
 \be
 \Str_0(\J):=\{\tau\in \Iso_{\R}(\J)| N(\tau A)=N(A)\}\\
\ee
\end{definition}

A cubic Jordan algebra element may be assigned a $\Str(\J)$ invariant \emph{rank} \cite{Jacobson:1961}.
\begin{definition}[Cubic Jordan algebra rank] A non-zero element $A\in\J$ has a rank given by:
\be
\begin{split}
\textrm{\emph{Rank}} A =1& \Leftrightarrow A^\sharp=0;\\
\textrm{\emph{Rank}} A =2& \Leftrightarrow N(A)=0,\;A^\sharp\not=0;\\
\textrm{\emph{Rank}} A =3& \Leftrightarrow N(A)\not=0.\\
\end{split}
\ee
\end{definition}

\subsection{Explicit Orbit Representatives of Lorentzian Spin Factors\label{sec:redJ}}

In this section we obtain the orbits representatives of the cubic Lorentzian spin factor Jordan algebras, which will be used in the subsequent treatment of the FTS orbits. The details of these algebras can be found in \cite{Gunaydin:2005zz}, where they were used in the context of generalised spacetimes and phase spaces.

Given  a vector space $V$ over a field $\F$ with a non-degenerate
quadratic form $Q(v), v\in V$, containing a base point $Q(c_0)=1$,
we may construct a  cubic Jordan algebra $\J_V=\F\oplus V$ with base
point $c=(1; c_0)\in \J_V$ and cubic norm, \be N(A)=aQ(v), \quad (a;
v)\in \J_V. \ee See, for example, \cite{McCrimmon:2004,
Krutelevich:2004}. In particular, the Lorentzian spin factors\footnote{In general, $\Gamma_{m,n}$ is a Jordan algebra with a
quadratic form of pseudo-Euclidean signature $\left( m,n\right) $,
\textit{i.e.} the Clifford algebra of $O\left( m,n\right) $
\cite{Jordan:1933vh}.}
$\mathfrak{J}_{1,n-1}:=\R\oplus\Gamma_{1,n-1}$ are defined by the
cubic norm, \be N(A)=aa_\mu a^\mu=a(a_{0}^{2}-a_ia_i), \quad
\textrm{where} \quad a\in\R\quad\textrm{and}\quad a_\mu\in\R^{1,n-1}
\ee for elements $A\in\mathfrak{J}_{1,n-1}$, \be A=(a; a_\mu)=(a;
a_0, a_i). \ee For notational convenience, we will often only write
the first three components  $(a; a_0, a_1)$ if $a_i=0$ for $i>1$.
The linearisation of the cubic norm is given by \be
N(A,B,C)=\frac{1}{3}(ab_\mu c^\mu+ca_\mu b^\mu+bc_\mu a^\mu). \ee
Evidently, $E=(1;1,0)$ is a base point and the corresponding Jordan
algebra maps are given by

\be\label{eq:cubicdefsspin}
\begin{split}
    \Tr(A)&=a+2a_0,\\
    S(A)&=2aa_0+a_\mu a^\mu,\\
    S(A, B)&=2(ab_0+ba_0+a_\mu b^\mu),\\
        \Tr(A, B)&=ab+2(a_0b_0+a_ib_i).
\end{split}
\ee
Using  $\Tr(A^\#,B)=3N(A,A,B)$ one obtains the quadratic adjoint \be
A^\#=(a_\mu a^\mu; aa^\mu), \ee where the index has been raised
using the mostly minus Lorentzian metric
$\eta^{\mu\nu}=\diag(1,-1,-1,\ldots,-1)$. Its linearisation $A\times
B=(A+B)^\#-A^\#-B^\#$ yields \be A\times B=(2a_\mu b^\mu;
ba^\mu+ab^\mu). \ee It is not difficult to verify that \be
A^{\#\#}=N(A)A, \ee so that the quadratic adjoint is indeed Jordan
cubic. Hence, we obtain a well defined Jordan algebra with Jordan
product defined by Eq. (\ref {eq:J3prod}), yielding \be A\circ
B=(ab; a_0b_0+\sum_j a_jb_j,a_0b_i+b_0a_i). \ee Three irreducible
idempotents are given by \be E_1=(1; 0),\quad E_2=(0;
\half,\half),\quad E_3=(0; \half,-\half). \ee The reduced structure
group $\Str_0(\Jnt)$ is given by $\SO(1,1)\times\SO(1,n-1)$, where
we have chosen to restrict to determinant 1 matrices.   Explicitly,
$A$ transforms as \be (a; a_\mu)\mapsto (e^{2\lambda}a;
e^{-\lambda}\Lambda_{\mu}{}^{\nu}a_\nu),\quad \text{where}\quad
\lambda\in\R, \Lambda\in\SO(1, n-1). \ee

\begin{theorem}\label{thm:d5red} For $n\geq 2$ every element $A=(a; a_\mu)\in\Jn$ of a given rank is $\SO(1,1)\times\SO(1,n-1)$ related to one of the following canonical forms:
\begin{enumerate}
\item Rank 1
\begin{enumerate}
\item  $A_{1a}=(1;0)=E_1$
\item $ A_{1b}=(-1;0)=-E_1$
\item $ A_{1c}=(0; \half, \half)=E_2$
\end{enumerate}
\item Rank 2
\begin{enumerate}
\item   $A_{2a}=(0; 1,0)=E_1+E_2$
\item  $ A_{2b}=(0; 0,1)=E_1-E_2$
\item  $ A_{2c}=(1; \half, \half)=E_1+E_2$
\item  $ A_{2d}=(-1; \half, \half)=-E_1+E_2$
\end{enumerate}
\item Rank 3
\begin{enumerate}
\item   $A_{3a}=(1; \half(1+k), \half(1-k))=E_1+E_2+kE_3$
\item   $A_{3b}=(-1; \half(1+k), \half(1-k))=-E_1+E_2+kE_3$
\end{enumerate}
\end{enumerate}
Note, if one restricts to the identity-connected component of
$\SO(1,n-1)$, each of the orbits $A_{1c}$, $A_{2c}$  and $A_{2d}$
splits into two cases, $A^{\pm}_{1c}$, $A^{\pm}_{2c}$ and
$A^{\pm}_{2d}$, corresponding to the future and past light cones.
Similarly, $A_{2a}$ splits into two disconnected components,
$A^{\pm}_{2a}$, corresponding to the future and past hyperboloids.
For $k>0$ the orbits $A_{3a}$ and $A_{3b}$ also split into
disconnected future and past hyperboloids, $A^{\pm}_{3a}$ and
$A^{\pm}_{3b}$.
\end{theorem}
\begin{proof}
$\text{Rank} A= 1 \Rightarrow$
\be
A^\sharp=(a_\mu a^\mu, aa^\mu)=0,\qquad (a; a_\mu)\not=0.
\ee
This corresponds to two cases: (i) $a_\mu=0, a\not=0$ or (ii) $a=0, a_\mu a^\mu=0, a_\mu\not=0$. In case (i) we have
\be
A=(a; 0)\mapsto (e^{2\lambda}a; 0)=(\pm1; 0).
\ee
In case (ii) we have
\be
A=(0; a_\mu)\mapsto(0; \Lambda_{\mu}{}^{\nu}a_\nu)= (0; \half,\half).
\ee

$\text{Rank} A= 2 \Rightarrow$
\be
N(A)=aa_\mu a^\mu=0,\qquad A^\sharp=(a_\mu a^\mu;  aa^\mu)\not=0.
\ee
This corresponds to two cases: (i) $a=0, a_\mu a^\mu\not=0$ or (ii) $a_\mu a^\mu=0, a\not=0, a_\mu\not=0$. In case (i) we have
\be
A=(0; a_\mu)\mapsto (0; e^{-\lambda}\Lambda_{\mu}{}^{\nu}a_\nu)=(0; 1,0)\quad\text{or}\quad(0; 0,1).
\ee
In case (ii) we have
\be
A=(a; a_\mu)\mapsto(e^{2\lambda}a; e^{-\lambda}\Lambda_{\mu}{}^{\nu}a_\nu)= (\pm1; \half,\half).
\ee

$\text{Rank} A= 3 \Rightarrow$
\be
N(A)=aa_\mu a^\mu\not=0.
\ee
Hence,
\be
A=(a; a_\mu)\mapsto (e^{2\lambda}a; e^{-\lambda}\Lambda_{\mu}{}^{\nu}a_\nu)=(\pm1; \half(1+k), \half(1-k)).
\ee
where $N(A)=\pm k$.
\qed\end{proof}
Note, the Lorentzian spin-factor  construction may be generalised to an arbitrary pseudo-Euclidean signature Jordan algebra $\J_{p-1, q-1}=\R\oplus\Gamma_{p-1, q-1}$  by defining the cubic norm,
\be\label{eq:n4gen}
N(A)=a a_{\mu}a^{\mu}, \quad(a; a_\mu)\in\J_{p-1, q-1},
\ee
where the index has been raised with a $\{+^{p-1}, -^{q-1}\}$ signature pseudo-Euclidean metric. The same base point $c=(1;1,0)$ may be used. A more ``democratic''  choice valid for any $p\geq 2$ is given by
\[c=\frac{1}{\sqrt{p-1}}(\sqrt{p-1}; \underbrace{1,1\ldots,1}_{p-1},\underbrace{0,0\ldots, 0}_{q-1}),\] although it obscures some of the symmetries by unnecessarily complicating the basic identities. In this case the reduced structure group is given by,
\be \Str_0(\J_{p-1,q-1})=\SO(1,1)\times\SO(p-1,q-1). \ee The
analysis goes through analogously to the Lorentzian case so we will
not treat it in detail here. See, for example,
\cite{Gunaydin:2009zza,ICL-2} for further details.
\subsubsection{Special Cases: $\J_{3\R}, \J_{2\R}$ and $\J_{\R}$\label{sec:d5special}}


\paragraph{Case 1: $\J_{3\R}$} The $n=2$ case $\J_{1,1}=\R\oplus\Gamma_{1,1}$ may be written as $\J_{3\R}=\R\oplus\R\oplus\R$. For $(a_1, a_2, a_3)\in\J_{3\R}$ and $(a; a_\nu)\in\J_{1,1}$ we have,
\be
a_1=a,
\quad a_2=a_0+a_1, \quad a_3=a_0-a_1, \ee so that the  cubic
norm takes the more democratic form \be N(A)=a_1a_2a_3.\ee
While the analysis follows that of the generic $n>2$ case, presented in \autoref{thm:d5red}, we highlight this form as it makes apparent the triality symmetry of the $n=2$ cubic norm, which we will return to in \autoref{sec:d4special}.  There are just three  irreducible idempotents:
\be
E_1=(1,0,0), \qquad E_2=(0,  1, 0),\qquad E_2=(0,  0, 1).
\ee
\paragraph{Case 2: $\J_{2\R}$} For $n=1$ the quadratic form on  $\Gamma_{1, n-1}$ becomes Euclidean and \autoref{thm:d5red} no longer holds.  $\J_{1,0}$ may be written as $\J_{2\R}=\R\oplus\R$. For $A=(a, a_0)\in\J_{2\R}$
 \be N(A)=a (a_0)^2.\ee
The symmetry of the cubic norm is reduced to $\SO(1,1)$ with a discrete factor, $\Z_2$. Note,  all rank 1 and 2 elements are respectively of the form $(a; 0)$ and $(0; a_0)$, where $a, a_0\not=0$. Consequently, unlike for $n>1$, $\J_{2\R}$ is \emph{not} spanned by its rank 1 elements. There are just two irreducible idempotents:
\be
E_1=(1;0), \qquad E_2=(-1; 1).
\ee
\paragraph{Case 3: $\J_{\R}$} The sequence may be completed by defining $\J_{\R}=\R$ with cubic norm,
 \be N(A)=A^3, \quad A\in \R. \ee
This cubic norm has no non-trivial symmetries. The unique choice of base point $c=1$ yields $\Tr(A)=3A, \Tr(A, B)=3AB, A^\sharp=A^2$, from which it is clear that the cubic norm is Jordan. Note, all non-zero elements are rank 3 and there are no irreducible idempotents.
\section{Orbits of Freudenthal Triple Systems \label{sec:FTS}}

The Freudenthal triple system  provides a  natural
representation of the dyonic black hole charge vectors for a broad
class of 4-dimesional supergravity theories, see for example \cite{Gunaydin:1983rk, Ferrara:1997uz,Gunaydin:2005gd,Gunaydin:2005mx,Gunaydin:2009dq,Borsten:2011ai}. In the following
treatment, we present the axiomatic definition of the FTS which is
manifestly covariant with respect to the 4-dimensional U-duality
group $G_4$. Subsequently, we present a particular realization in
terms of Jordan algebras. This is equivalent to decomposing $G_4$
with respect to 5-dimensional U-duality group $G_5$, which is
modeled by the corresponding Jordan algebra. Consequently, this
particular realization is manifestly covariant with respect to
$G_5$.

\subsection{Axiomatic Definition of The FTS\label{sec:FTS-axio}}

An FTS may be axiomatically defined \cite{Brown:1969} as a finite
dimensional vector space $\FTS$ over a field $\F$ (not of
characteristic 2 or 3), such that:
\begin{enumerate}
\item  $\FTS$ possesses a non-degenerate antisymmetric bilinear form $\{x, y\}.$
\item $\FTS$ possesses a symmetric four-linear form $q(x,y,z,w)$ which is not identically zero.
\item If the ternary product $T(x,y,z)$ is defined on $\FTS$ by $\{T(x,y,z), w\}=q(x, y, z, w)$, then
\[3\{T(x, x, y), T(y,y,y)\}=\{x, y\}q(x, y, y, y).\]
\end{enumerate}

\subsection{Definition Over a Cubic Jordan Algebra\label{sec:FTS-J}}

Given a cubic Jordan algebra $\mathfrak{J}$ defined over a field
$\R$, there exists a corresponding  FTS
\begin{equation}
\mathfrak{F}(\mathfrak{J})=\mathds{R\oplus R}\oplus
\mathfrak{J}\oplus\mathfrak{J}.
\end{equation}
An arbitrary element $x\in \mathfrak{F}(\mathfrak{J})$ may be
written as a ``$2\times 2$ matrix''
\begin{equation}
x=\begin{pmatrix}\alpha&A\\B&\beta\end{pmatrix}, \quad\text{where}
~\alpha, \beta\in\R\quad\text{and}\quad A, B\in\mathfrak{J}.
\end{equation}
The FTS comes equipped with a non-degenerate bilinear antisymmetric quadratic form, a quartic form and a trilinear triple product:
\begin{subequations}
\begin{enumerate}
\item Quadratic form $ \{\bullet, \bullet\}$: $\mathfrak{F}(\mathfrak{J})\times\mathfrak{F}(\mathfrak{J})\to \R$
    \begin{equation}\label{eq:bilinearform}
        \{x, y\}:=\alpha\delta-\beta\gamma+\Tr(A,D)-\Tr(B,C),    \text{\qquad where\qquad}x=\begin{pmatrix}\alpha&A\\B&\beta\end{pmatrix},\; y=\begin{pmatrix}\gamma&C\\D&\delta\end{pmatrix}.
        \end{equation}
\item Quartic form $\Delta:\mathfrak{F}(\mathfrak{J})\to \R$
    \begin{equation}\label{eq:quarticnorm}
    \Delta (x):=-(\alpha\beta-\Tr(A,B))^2-4[\alpha N(A)+\beta N(B)-\Tr(A^\sharp, B^\sharp)]=:\frac{1}{2}q(x).
    \end{equation}

\item Triple product $T:\mathfrak{F}(\mathfrak{J})\times
\mathfrak{F}(\mathfrak{J})\times\mathfrak{F}(\mathfrak{J})\to\mathfrak{F}(\mathfrak{J})$ which is uniquely defined by
\begin{equation}
\{T(x, y, w), z\}=2\Delta(x, y, w, z),
\end{equation}
where $\Delta(x, y, w, z)$ is the full linearization of $\Delta(x)$ normalized such that $\Delta(x, x, x, x)=\Delta(x)$.
\end{enumerate}
\end{subequations}
Explicitly, the triple product is given by
\begin{equation}\label{eq:Tofx}
T(x)=\left(\begin{array}{cc}-\alpha^2\beta+\Tr(A,B)-N(B)&-(\beta B^\sharp-B\times A^\sharp)+(\alpha\beta-\Tr(A,B))A\\
(\alpha A^\sharp-A\times B^\sharp)-(\alpha\beta-\Tr(A,B))B&\alpha\beta^2-\Tr(A,B)+N(A)\end{array}\right).
\end{equation}
Note that \textit{all} the necessary definitions, such as the cubic
and trace bilinear forms, are inherited from the underlying Jordan
algebra $\mathfrak{J}$.

\begin{remark} For the Jordan algebras introduced in \autoref{sec:J}, we will  use the short hand notation:
\[
\begin{array}{llllllll}
\Fnt:=\FTS(\Jnt),&\Fnf:=\FTS(\Jnf),&\FTS_{3\R}:=\FTS(\J_{3\R}),&\FTS_{2\R}:=\FTS(\J_{2\R}),&\FTS_{\R}:=\FTS(\J_{\R}).
\end{array}
\]
\end{remark}

\subsection{The Automorphism Group}\label{sec:FTS-auto}

\begin{definition}[The automorphism group $\Aut(\FTS)$] The \emph{automorphism} group of an FTS is defined  as the set of invertible  $\R$-linear transformations preserving the quartic and quadratic forms:
\be
\Aut(\FTS):=\{\sigma\in\Iso_\R(\FTS)|\{\sigma x, \sigma y\}=\{x, y\},\;\Delta(\sigma x)=\Delta(x)\}\label{eq:brownfts}.
\ee
\end{definition}
Note, the conditions $\{\sigma x, \sigma y\}=\{x, y\}$ and $\Delta(\sigma x)=\Delta(x)$ immediately imply
$
 T(\sigma x)=\sigma T(x).
$

For $\FTS(\J_{3}^{\alg^{(s)}})$ the automorphism group has a two element
centre and its quotient yields the simple groups listed in
\autoref{tab:FTSsummary}, while for $\Fnt$ one obtains the
semi-simple groups $\SL(2, \R)\times\SO(2, n)$
\cite{Brown:1969,Krutelevich:2004, Gunaydin:2009zza}. In all cases
$\FTS$ forms a symplectic representation of $\Aut{(\FTS)}$, the
dimensions of which are listed in the final column of
\autoref{tab:FTSsummary}. This table covers a number 4-dimensional
supergravities: $\Fnt, \Fnf\rightarrow\mathcal{N}=2,4$
Maxwell-Einstein supergravity, $\FTS(\J_{3}^{\alg})\rightarrow\mathcal{N}=2$
``magic'' Maxwell-Einstein supergravity and
$\FTS(\J_{3}^{\Oct^{(s)}})\rightarrow \mathcal{N}=8$ maximally supersymmetric
supergravity (see, for example,
\cite{Gunaydin:1983bi,Gunaydin:1983rk,Rios:2007qn,Borsten:2008wd,Borsten:2009zy,Borsten:2010aa,Rios:2010br,ICL-2,Rios:2011fa}).
Moreover, the special case of $\mathfrak{F}_{3\R}$ (and its
generalisations) has found applications in the theory of
entanglement \cite{Borsten:2008,levay-2008, Borsten:2009yb,Levay:2009, Borsten:2010ths}.

\begin{table}
\caption[Jordan algebras, corresponding FTSs, and their associated
symmetry groups]{The automorphism  group
$\Aut(\mathfrak{F}(\mathfrak{J}))$ and the dimension of its
representation $\dim\mathfrak{F}(\mathfrak{J})$ given by the
Freudenthal construction defined over the cubic Jordan algebra
$\mathfrak{J}$ (with dimension $\dim\mathfrak{J}$ and reduced
structure group $\Str_0(\mathfrak{J})$). Here $\J_{3}^{\alg^{(s)}}$ denotes the cubic Jordan algebra of $3 \times 3$ Hermitian matrices over the (split) composition algebras.  \label{tab:FTSsummary}}
\begin{tabular*}{\textwidth}{@{\extracolsep{\fill}}c*{5}{M{c}}c}
\toprule
& \text{Jordan algebra }\mathfrak{J} & \Str_0(\mathfrak{J}) & \dim\mathfrak{J} & \Aut(\mathfrak{F}(\mathfrak{J})) & \dim\mathfrak{F}(\mathfrak{J}) &\\
\hline
& \R                     & -                                  & 1   & \SL(2,\R)                                   & 4    & \\
& \R\oplus\R           & \SO(1,1)                       & 2   & \SL(2,\R)\times \SL(2,\R)                  & 6    & \\
& \R\oplus\R\oplus\R & \SO(1,1)\times \SO(1,1)    & 3   & \SL(2,\R)\times \SL(2,\R)\times \SL(2,\R) & 8    & \\
& \R\oplus \Gamma_{1,n-1}           & \SO(1,1)\times \SO(1,n-1)  & n+1 & \SL(2,\R)\times \SO(2,n)                & 2(n+2) & \\
& \R\oplus \Gamma_{5,n-1}           & \SO(1,1)\times \SO(5,n-1)  & n+5 & \SL(2,\R)\times \SO(6,n)                & 2(n+6) & \\
& \J_{3}^{\R}             & \SL(3, \R)                         & 6   & \Sp(6,\R)                                   & 14   & \\
& \J_{3}^{\C}             & \SL(3,\C)                         & 9   & \SU(3,3)                                 & 20   & \\
& \J_{3}^{\C^s}             & \SL(3,\R)\times\SL(3, \R)                         & 9   & \SL(6, \R)                                 & 20   & \\
& \J_{3}^{\Q}             & \SU^\star(6)                   & 15  & \SO^\star(12)                            & 32   & \\
& \J_{3}^{\Q^s}             & \SL(6, \R)                   & 15  & \SO(6,6)                            & 32   & \\
& \J_{3}^{\Oct}             & E_{6(-26)}                   & 27  & E_{7(-25)}                            & 56   & \\
& \J_{3}^{\Oct^s}             & E_{6(6)}                   & 27  & E_{7(7)}                            & 56   & \\
\bottomrule
\end{tabular*}
\end{table}

\begin{lemma}The Lie algebra $\mathfrak{Aut}(\mathfrak{F})$ of $\Aut(\FTS)$ is given by
\be \mathfrak{Aut}(\mathfrak{F})=\{\phi\in
\Hom_{\R}(\mathfrak{F})|\Delta (\phi x,x,x,x)=0,\{\phi x, y\}+\{x,
\phi y\}=0, \;\forall x,y\in \mathfrak{F}\}. \ee
\end{lemma}
\begin{proof} If $\phi\in\Hom_{\R}(\mathfrak{F})$ satisfies $\Delta (e^{t\phi}x, e^{t\phi}x, e^{t\phi}x, e^{t\phi}x)=\Delta (x,x,x,x)$, where $t\in \R$,   differentiating with respect to $t$ and then setting $t=0$ one obtains  $\Delta (\phi x,x,x,x)=0$. Similarly, if $\{e^{t\phi} x, e^{t\phi} y\}=\{x,y\}$, then $\{\phi x, y\}+\{x, \phi y\}=0$. Conversely, assuming $\{\phi x, y\}+\{x, \phi y\}=0$ for all $x,y\in\mathfrak{F}$ let $\sigma=e^{t\phi}$. Then,
\be
\begin{split}
\{e^{t\phi} x, e^{t\phi} y\}&=\{(1+t\phi+\half t^2\phi^2+\ldots)x,(1+t\phi+\half t^2\phi^2+\ldots)y\}\\
&=\{x,y\}+t(\{\phi x, y\}+\{x, \phi y\})\\
&\phantom{=}+t^2(\half\{\phi x, \phi y\}+\half\{\phi^2 x,  y\}+\half\{\phi x, \phi y\}+\half\{ x, \phi^2 y\})+\dots\\
&=\{x,y\}.
\end{split}
\ee Similarly, assuming $\Delta (\phi x,x,x,x)=0$ and letting
$\sigma=e^{t\phi}$, then $\Delta (\sigma x)=\Delta (x)$.
\qed\end{proof}

The automorphism group may also be defined in terms of the  \emph{Freudenthal product} \cite{Shukuzawa:2006,Yokota:2009}. The Freudenthal product is useful in that it can be used to form elements of the Lie algebra and, further, we will need it to distinguish the orbits.

\begin{definition}[Freudenthal product] For $x=(\alpha, \beta, A, B),\; y=(\delta,\gamma,C,D)$, define the Freudenthal product
\[
\wedge:\FTS\times\FTS\rightarrow\Hom_\R(\FTS)
\]
by, \be x\wedge y:=\Phi(\phi, X, Y, \nu), \quad\textrm{where}\quad
\left\{ \begin{array}{lll}
\phi    &=&-(A\vee D+B\vee C)\\
X           &=&-\half(B\times D-\alpha C-\delta A)\\
Y           &=&\half(A\times C-\beta D-\gamma B)\\
\nu     &=&\frac{1}{4}(\Tr(A,D)+\Tr(C,B)-3(\alpha\gamma+\beta\delta)),\\
\end{array}\right.
\ee
and $A\vee B\in\Rstr{\J}$ is defined by $(A\vee B)C=\half\Tr(B,C)A+\frac{1}{6}\Tr(A,B)C-\half B\times(A\times C)$.  The action of $\Phi:\FTS  \rightarrow\FTS$ is given by
\be\label{eq:ftslieaction}
\Phi(\phi,X,Y,\nu)\begin{pmatrix}\alpha&A\\B&\beta\end{pmatrix}=\begin{pmatrix}\alpha\nu+(Y,B)&\phi A-\frac{1}{3}\nu A+2Y\times B +\beta X \\-^t\phi B+\frac{1}{3}\nu B+2X\times A+\alpha Y&-\beta\nu+(X,A)\end{pmatrix}.
\ee
\end{definition}

\begin{lemma} The automorphism group is given by the set of invertible  $\R$-linear transformations preserving the Freudenthal product:
\be
\Aut{\FTS}=\{\sigma\in\Iso_\R(\FTS) | \sigma(x\wedge y)\sigma^{-1}=\sigma x\wedge\sigma y\}.
\ee
\end{lemma}
\begin{proof} We proceed by establishing the equivalence
\be
\sigma(x\wedge y)\sigma^{-1}=\sigma x\wedge\sigma y\Leftrightarrow \{\sigma x, \sigma y\}=\{x, y\},\;\Delta(\sigma x)=\Delta(x).
\ee
We begin with the right implication $\Rightarrow$. First, following \cite{Yokota:2009}, we show that
$\sigma(x\wedge y)\sigma^{-1}=\sigma x\wedge\sigma y\Rightarrow \{\sigma x, \sigma y\}=\{x, y\}$.
Using the identity   \cite{Yokota:2009},
\be\label{eq:yok}
(x\wedge y)x-(x\wedge x)y+\frac{3}{8}\{x,y\}x=0,
\ee
 it follows that $(\sigma x\wedge \sigma y)\sigma x-(\sigma x\wedge \sigma x)\sigma y+\frac{3}{8}\{\sigma x,\sigma y\}\sigma x=0$, which, from our assumption $\sigma(x\wedge y)\sigma^{-1}=\sigma x\wedge\sigma y$ implies
\be
\begin{split}
 &\sigma( x\wedge  y) x-\sigma( x\wedge  x) y+\frac{3}{8}\{\sigma x,\sigma y\}\sigma x=0\\
\Rightarrow &\sigma(-\frac{3}{8}\{x,y\}x )+\frac{3}{8}\{\sigma x,\sigma y\}\sigma x=0\\
\Rightarrow &\frac{3}{8}(\{\sigma x,\sigma y\}-\{x,y\})\sigma x=0\\
\Rightarrow &\{\sigma x,\sigma y\}=\{x,y\}.
\end{split}
\ee
Since $\frac{4}{3}(x\wedge x)x=T(x)$, $\sigma(x\wedge y)\sigma^{-1}=\sigma x\wedge\sigma y$ implies $T(\sigma x)=\sigma T(x)$ and therefore
$
\{\sigma x, \sigma y\}=\{x, y\}\Rightarrow \Delta(\sigma x)=\Delta(x).
$

To establish the left implication $\Leftarrow$, we begin by noting that  $T(x, y, z)=\frac{4}{9}[(x\wedge y)z+(y\wedge z)x+(z\wedge x)y]$. Then from the identity,
\be\label{eq:yok2}
2(x\wedge y)z-(x\wedge z)y-(y\wedge z)x+\frac{3}{8}\{z,y\}x-\frac{3}{8}\{x,z\}y=0,
\ee
which is easily obtained from \eqref{eq:yok}, see Lemma 4.1.1 of \cite{Yokota:2009}, we have,
\be\label{eq:yok3}
-\frac{9}{4}T(x,y,z)+3(x\wedge y)z+\frac{3}{8}\{z,y\}x-\frac{3}{8}\{x,z\}y=0.
\ee
Since $\sigma T(x,y,z)=T(\sigma x,\sigma y, \sigma z)$, our assumption $\{\sigma x, \sigma y\}=\{x, y\}$ together with \eqref{eq:yok3} implies
\be\label{eq:yok4}
\sigma [-\frac{9}{4}T(x,y,z)+\frac{3}{8}\{z,y\}x-\frac{3}{8}\{x,z\}y]+3(\sigma x\wedge \sigma y)\sigma z=0
\ee 
and so, substituting $-3(x\wedge y)\sigma^{-1}\sigma z$ in the square parentheses of \eqref{eq:yok4}, 
\be
3[(\sigma x\wedge \sigma y)-\sigma( x\wedge y)\sigma^{-1}]\sigma z=0
\ee 
implying
$
(\sigma x \wedge \sigma y)=\sigma ( x \wedge  y)\sigma^{-1},
$
as required.
\qed\end{proof}
Finally, we recall  the explicit elements of $\Aut(\FTS)$:

\begin{lemma}[Seligman, 1962; Brown, 1969 \cite{Seligman:1962, Brown:1969}] The following transformations generate elements of $\Aut(\FTS)$:
\be\label{eq:ftstrans}
\begin{split}
\varphi(C):\pmtwo{\alpha}{A}{B}{\beta}&\mapsto \pmtwo{\alpha+(B,C)+(A,C^\sharp)+\beta N(C)}{A+\beta C}{B+A\times C +\beta C^\sharp}{\beta};\\
\psi(D):\pmtwo{\alpha}{A}{B}{\beta}&\mapsto \pmtwo{\alpha}{A+B\times D +\alpha D^\sharp}{B+\alpha D}{\beta+(A,D)+(B,D^\sharp)+\alpha N(D)};\\
\widehat{\tau}:\pmtwo{\alpha}{A}{B}{\beta}&\mapsto \pmtwo{\lambda^{-1}\alpha}{\tau A}{^t\tau^{-1} B}{\lambda\beta};
\end{split}
\ee
where $C,D\in\J$ and $\tau\in \Str(\J)$ s.t. $N(\tau A)=\lambda N(A)$.
\end{lemma}
 For convenience we further define $\mathcal{Z}:=\phi(-\mathds{1})\psi(\mathds{1})\phi(-\mathds{1})$,
$
\mathcal{Z}:\pmtwo{\alpha}{A}{B}{\beta}\mapsto \pmtwo{-\beta}{-B}{A}{\alpha}.
$

 \subsubsection{Rank Conditions and $\Aut(\FTS)$-Equivalence}

\begin{lemma}[Krutelevich 2004] Every non-zero element of $\FTS(\J)$, where $\J$ is one of $\JA, \JAs, \Jnt$ or  $\J_{3\R}, \J_{2\R}, \J_{\R}$, can be brought by $\Aut(\FTS)$ into the reduced form
\be\label{eq:redform}
x_{\text{red}}=\begin{pmatrix}1&A\\0&\beta\end{pmatrix}.
\ee
\end{lemma}
\begin{proof}
The proof presented by Krutelevich \cite{Krutelevich:2004} for
$\JAs$ also holds for $\JA, \Jnt, \J_{3\R}$, since it only requires
that $\J$ is spanned by its rank 1 elements, and that  the bilinear
trace form is non-degenerate. However,  $\J_{2\R}$ and $\J_{\R}$ are
not spanned by their rank 1 elements\footnote{Indeed, $\J_{\R}$ has
no rank 1 elements.} so a gentle modification of the proof is
required. It is sufficient to show that the theorem holds for
$\J_{\R}$, so we focus on this case. Starting from an arbitrary
element
\[
\begin{pmatrix}\alpha&A\\B&\beta\end{pmatrix},
\]
we first show that one may always assume $\beta\not=0$. Assume
$\beta=0$. If $\alpha\not=0$  we simply apply
$\phi(-1)\psi(1)\phi(-1)$. Now, assume $\alpha, \beta=0$. This
implies we may assume at least one of $A$ or $B$ are non-zero.
Applying $\psi(D)$ we find
\[
\beta=0\mapsto \beta'=3D(BD+A)
\]
so that we can always pick a $D$ such that $\beta'\not=0$. Hence, we
may now assume from the outset that $\beta\not=0$. We now proceed by
illustrating that we may always assume $\alpha=1$. Assume
$\beta\not=0$ and apply
\[
\phi(C): \alpha\mapsto \beta C^3+3AC^2+3BC+\alpha.
\]
Since
\[
\beta C^3+3AC^2+3BC+\alpha-1=0
\]
has at least one real root for $\beta\not=0$ we have established
that we may assume $\alpha=1$. To finish the proof we assume
$\alpha=1$, and apply $\psi(-B)$.
\qed\end{proof}
\begin{lemma}\label{lem:rank3}  \emph{(i)} An element of $\FTS(\J)$, where $\J$ is one of $\JA, \JAs, \Jnt$, of the form
\be
\begin{pmatrix}\alpha&a_iE_i\\0&\beta\end{pmatrix}, \quad i=1,2,3,
\ee
is $\Aut(\FTS)$ related to:
\begin{enumerate}
\item
\be
\begin{pmatrix}\alpha&(a_1+\beta c-\frac{a_2a_3c^2}{\alpha})E_1+a_2E_2+a_3E_3\\0&\beta-\frac{2a_2a_3c}{\alpha}\end{pmatrix}.
\ee
\item
\be
\begin{pmatrix}\alpha&a_1E_1+(a_2+\beta c-\frac{a_1a_3c^2}{\alpha})E_2+a_3E_3\\0&\beta-\frac{2a_1a_3c}{\alpha}\end{pmatrix}.
\ee
\item
\be
\begin{pmatrix}\alpha&a_1E_1+a_2E_2+(a_3+\beta c-\frac{a_1a_2c^2}{\alpha})E_3\\0&\beta-\frac{2a_1a_2c}{\alpha}\end{pmatrix}.
\ee
\end{enumerate}
\emph{(ii)} An element of $\FTS_{2\R}$,  of the form
\be\label{eq:f2rlem27}
\begin{pmatrix}\alpha&a_iE_i\\0&\beta\end{pmatrix} \quad i=1,2,
\ee
is $\Aut(\FTS)$ related to:
\be
\begin{pmatrix}\alpha&(a_1+\beta c)E_1+a_2E_2\\0&\beta-\frac{2c(a_2)^2}{\alpha}\end{pmatrix}.
\ee
\end{lemma}
\begin{proof}
(i) The proof presented by Krutelevich \cite{Krutelevich:2004} for $\JA, \JAs$   may be seen to hold for $\Jnt$ by direct calculation. (ii) Follows by acting on \eqref{eq:f2rlem27} with
\be
\psi(D)\circ\phi(C),
\ee
where $C=(c;0)$ and $D=-\frac{1}{\alpha}(0; ca_2)$.
\qed\end{proof}
Following \cite{Krutelevich:2004} one may generalise the
conventional matrix rank to the FTS.
\begin{definition}[FTS ranks] An FTS element may be assigned an $\Aut(\FTS)$ invariant \emph{rank}:
\be\label{eq:FTSranks}
\begin{split}
\textrm{\emph{Rank}} x =1& \Leftrightarrow \Upsilon(x,x,y)=0\; \forall y,\; x\not=0;\\
\textrm{\emph{Rank}} x =2& \Leftrightarrow T(x)=0,\;\exists y \;\text{s.t.}\;\Upsilon(x,x,y)\not=0;\\
\textrm{\emph{Rank}} x =3& \Leftrightarrow \Delta(x)=0,\;T(x)\not=0;\\
\textrm{\emph{Rank}} x =4& \Leftrightarrow \Delta(x)\not=0,\\
\end{split}
\ee
where we have defined $\Upsilon(x,x,y):=3T(x,x,y)+\{x, y\}x$.
\end{definition}
\begin{remark}[Reduced rank conditions]  For an element in the reduced form \eqref{eq:redform} the rank conditions simplify:
\be
\begin{split}
\textrm{\emph{Rank}} x =1& \Leftrightarrow A=0,\; \beta=0;\\
\textrm{\emph{Rank}} x =2&\Leftrightarrow A^\sharp=0,\;\beta=0,\;A\not=0;\\
\textrm{\emph{Rank}} x =3& \Leftrightarrow 4N(A)=-\beta^2,\;A^\sharp\not=0;\\
\textrm{\emph{Rank}} x =4& \Leftrightarrow 4N(A)\not=-\beta^2.\\
\end{split}
\ee
\end{remark}
In order to distinguish orbits of the same rank we will use the following quadratic form introduced in \cite{Shukuzawa:2006}.
\begin{definition}[FTS quadratic form] Define, for a non-zero constant element $y\in\FTS$, the real quadratic form $B_y$,
\be
 B_y(x):=\{(x\wedge x)y, y\}, \qquad x \in \FTS.
\ee
\end{definition}
\begin{lemma}[Shukuzawa, 2006 \cite{Shukuzawa:2006}] If $y'=\sigma y$ for $y\not=0$ and $\sigma\in \Aut(\FTS)$, then
\[
B_y(x)=B_{y'}(x'),\quad\textrm{where}\quad x'=\sigma x \in \FTS.
\]
\end{lemma}

\subsection{Explicit Orbit Representatives of Reducible FTS $\Fnt$}\label{sec:Spin-Factors}

In this section we obtain the explicit orbit representatives of the reducible FTS defined over the  Lorentzian spin factor Jordan algebras. These FTS have previously appeared in physics literature  \cite{Gunaydin:2005zz,Gunaydin:2009dq}, where, in particular,   quasiconformal realisations over the FTS where used to capture the three dimensional U-duality groups as spectrum generating quasiconformal groups.

\begin{theorem}\label{thm:redcanforms} Every element $x\in\Fnt$ with $n\geq 2$ of a given rank is $\SL(2, \R)\times\SO(2,n)$ related to one of the following canonical forms:
\begin{enumerate}
\item Rank 1
\begin{enumerate}
\item  $x_{1}=\begin{pmatrix}1&0\\0&0\end{pmatrix}$
\end{enumerate}
\item Rank 2
\begin{enumerate}
\item  $x_{2a}=\begin{pmatrix}1&(1;0,0)\\0&0\end{pmatrix}$
\item  $x_{2b}=\begin{pmatrix}1&(-1;0,0)\\0&0\end{pmatrix}$
\item  $x_{2c}=\begin{pmatrix}1&(0;\half,\half)\\0&0\end{pmatrix}$
\end{enumerate}
\item Rank 3
\begin{enumerate}
\item  $x_{3a}=\begin{pmatrix}1&(0;1,0)\\0&0\end{pmatrix}$
\item  $x_{3b}=\begin{pmatrix}1&(0;0,1)\\0&0\end{pmatrix}$
\end{enumerate}
\item Rank 4
\begin{enumerate}
\item   $x_{4a}=k\begin{pmatrix}1& (-1;1,0)\\0&0\end{pmatrix}$
\item   $x_{4b}=k\begin{pmatrix}1&(1;0,1)\\0&0\end{pmatrix}$
\item   $x_{4c}=k\begin{pmatrix}1&(-1;0,1)\\0&0\end{pmatrix},$
\end{enumerate}
where $k>0$.
\end{enumerate}
\end{theorem}
\begin{proof}
We begin by transforming to the generic canonical form \eqref{eq:redform}
and then we proceed case by case, according to the rank.

\paragraph{Rank 1:}
$ \text{Rank} x=1\Rightarrow A=0,\beta=0, $ so that every rank 1
element is $\Aut(\Fnt)$ related to \be
x_1=\begin{pmatrix}1&0\\0&0\end{pmatrix}. \ee

\paragraph{Rank 2:}
$ \text{Rank} x=2\Rightarrow A^\sharp=0,\beta=0, A\not=0, $ so that
every rank 2 element is $\Aut(\Fnt)$ related to \be
x=\begin{pmatrix}1&A\\0&0\end{pmatrix}, \ee where $A$ is a rank 1
Jordan algebra element. $A$ may be brought into  canonical form via
$\hat{\tau}$, where $\tau\in\Str_0(\Jn)$,
\[
 \widehat{\tau}:\pmtwo{1}{A}{0}{0}\mapsto \pmtwo{1}{\tau A}{0}{0}.
 \]
 Hence, $x$ may be brought into one of three forms corresponding to the three rank 1 representatives for
 $A$, namely:
\be x_{2a}=\begin{pmatrix}1&A_{1a}\\0&0\end{pmatrix};\qquad
x_{2b}=\begin{pmatrix}1&A_{1b}\\0&0\end{pmatrix};\qquad
x_{2c}=\begin{pmatrix}1&A_{1c}\\0&0\end{pmatrix}. \ee These are in
fact unrelated, as it can be seen by computing the quadratic forms
\be\label{eq:signature1}
\begin{split}
B_{x_{2a}}(y)&=-c_\mu c^\mu+\gamma d;\\
B_{x_{2b}}(y)&=\phantom{-}c_\mu c^\mu-\gamma d;\\
B_{x_{2c}}(y)&=-cc_0-c c_1+\gamma d_0+\gamma d_1,
\end{split}
\ee where \be y=\begin{pmatrix}\delta&C\\D&\gamma\end{pmatrix}. \ee
By diagonalising \eqref{eq:signature1}, one can verify that the
three forms have distinct signatures; hence, by Sylvester's Law of
Inertia, $x_{2a}, x_{2b}$ and $x_{2c}$ lie in distinct orbits.

\paragraph{Rank 3:}
$ \text{Rank} x=3\Rightarrow N(A)=-\frac{\beta^2}{4},
A^\sharp\not=0. $ If $\beta\not=0$ then $A$ is $\Str_0(\Jn)$ related
to \be (\pm1; \frac{1}{2}(1\mp\frac{\beta^2}{4}),
\frac{1}{2}(1\pm\frac{\beta^2}{4}), 0,\ldots)=\pm
E_1+E_2\mp\frac{\beta^2}{4}E_3, \ee so that, by an application of
$\hat{\tau}$, one obtains \be x=\begin{pmatrix}1&\pm
E_1+E_2\mp\frac{\beta^2}{4}E_3\\0&\beta\end{pmatrix}. \ee Then, by
Lemma \ref{lem:rank3} with $c=\mp\frac{\beta^2}{4}$, $x$ may be
brought into the form, \be
\begin{pmatrix}1&\pm E_1+E_2\\0&0\end{pmatrix}.
\ee Hence, we may assume from the out set that \be
x=\begin{pmatrix}1&A\\0&0\end{pmatrix} \ee where $A$ is a rank 2
Jordan algebra element. Via an application of $\hat{\tau}$, where
$\tau\in\Str_0(\Jn)$, $x$ may be brought into one of four forms
corresponding to the four rank 2 canonical forms for $A$, namely:
\be x_{3a}=\begin{pmatrix}1&A_{2a}\\0&0\end{pmatrix};\qquad
x_{3b}=\begin{pmatrix}1&A_{2b}\\0&0\end{pmatrix};\qquad
x_{3c}=\begin{pmatrix}1&A_{2c}\\0&0\end{pmatrix};\qquad
x_{3d}=\begin{pmatrix}1&A_{2d}\\0&0\end{pmatrix}. \ee We are now
able to show $x_{3a}$ and $x_{3b}$ are $\Aut(\Fnt)$ related to
$x_{3d}$ and $x_{3c}$ respectively. The proof proceeds by an
application of $ \varphi(\tilde{C})\psi(D)\varphi(C) $ with
$\tilde{C}^\sharp=D^\sharp=C^\sharp=0$, which yields,
\[
\begin{pmatrix}1&A_{2a}\\0&0\end{pmatrix}\mapsto \begin{pmatrix}1+(A_{2a}\times C+D,\tilde{C})&A_{2a}+(A_{2a}\times C)\times D+(A_{2a}, D)\tilde{C}\\A_{2a}\times C+D+(A_{2a}+(A_{2a}\times C)\times D)\times \tilde{C}&(A_{2a}, D)\end{pmatrix}.
\]
Assuming  $(A_{2a}, D)=1$ and $\tilde{C}=-(A_{2a}+(A_{2a}\times
C)\times D)$, one obtains \be
\begin{pmatrix}1-(A_{2a}\times C+D,A_{2a}+(A_{2a}\times C)\times D)&0\\A_{2a}\times C+D&1\end{pmatrix}.
\ee This is achieved by the choice $C=(0;-\half,-\half,0,\ldots)$
and $D=(0;\half,\half,0,\ldots)$. One finds
$\tilde{C}=(0;-\half,-\half,0,\ldots)$ and $(A_{2a}\times
C+D,\tilde{C})=-1$, yielding
\[
\begin{pmatrix}0&0\\(-1;\half,\half,0,\ldots)&1\end{pmatrix},
\]
from which, after three applications of
$\varphi(-\mathds{1})\psi(\mathds{1})\varphi(-\mathds{1})$,  the desired form \be
\begin{pmatrix}1&A_{2d}\\0&0\end{pmatrix},
\ee follows. Similary, it can be proved that $x_{2b}$ is $\Aut(\Fnt)$ related
to $x_{2c}$.

The remaining two possiblilities are unrelated, as it can be seen by
computing the quadratic forms \be\label{eq:signature2}
\begin{split}
B_{x_{3c}}(y)&=-\delta c_0-\delta c_1-c c_0-c c_1-c_\mu c^\mu+\gamma d+\gamma d_0+\gamma d_1+d d_0+d d_1;\\
B_{x_{3d}}(y)&=\phantom{-}\delta c_0+\delta c_1-c c_0-c c_1+c_\mu c^\mu-\gamma d+\gamma d_0+\gamma d_1-d d_0-d d_1.\\
\end{split}
\ee By diagonalising \eqref{eq:signature2}, one can verify that the
two forms have distinct signatures; hence, by Sylvester's Law of
Inertia, $x_{2a}$ and $x_{2b}$ lie in distinct orbits.

\paragraph{Rank 4:}
$ \text{Rank} x=4\Rightarrow
\Delta(x)=-N(A)-\frac{\beta^2}{4}\not=0. $ By Lemma \ref{lem:rank3},
we may assume from the out set that \be
x=\begin{pmatrix}1&A\\0&0\end{pmatrix} \ee where $A$ is a rank 3
Jordan algebra element. Via an application of $\hat{\tau}$, where
$\tau\in\Str_0(\Jn)$, $x$ may be brought into one of two forms
corresponding to the two rank 3 canonical forms for $A$, namely: \be
x_{4a}=\begin{pmatrix}1&(1;\half(1-m),\half(1+m),0,\ldots)\\0&0\end{pmatrix},\qquad
x_{4b}=\begin{pmatrix}1&(-1;\half(1+m),\half(1-m),0,\ldots)\\0&0\end{pmatrix},
\ee where conventions have been chosen such that
$\Delta(x_{4a})=\Delta(x_{4b})=m$.

In order to determine under what conditions $x_{4a}$ and $x_{4b}$
are related, use the  quadratic forms \be\label{eq:signature3}
\begin{split}
B_{x_{4a}}(y)=&\phantom{-}\delta m c - \delta c_{0} + \delta m c_{0} - c c_{0} + m c c_{0}  - \delta c_{1} - \delta m c_{1} - c c_{1} -
 m c c_{1} - c_{\mu}c^{\mu} \\
 &+ \gamma d + \gamma d_{0} - \gamma m d_{0} + d d_{0} - m d d_{0} - m d_{\mu}d^{\mu} +
  \gamma d_{1} + \gamma m d_{1} + d d_{1} + m d d_{1};\\
B_{x_{4b}}(y)=&-\delta m c + \delta c_{0} + \delta m c_{0} - c c_{0} - m c c_{0}  + \delta c_{1} - \delta m c_{1} - c c_{1} +
 m c c_{1} - c_{\mu}c^{\mu} \\
 &- \gamma d + \gamma d_{0} + \gamma m d_{0} - d d_{0} - m d d_{0} - m d_{\mu}d^{\mu} +
  \gamma d_{1} - \gamma m d_{1} - d d_{1} + m d d_{1}.\\
\end{split}
\ee is made once again.

The diagonalisation of Eq. \eqref{eq:signature3} leads to quite
complicated expressions for the two metrics. However, one can show
that they only differ in three components, namely
$(1,-\frac{m}{2},m)$ \textit{versus} $(-1,-m,-\frac{m}{2})$; hence,
one can conclude that for $m>0$ the metrics have different
signatures. Consequently, for $m>0$, $x_{4a}$ and $x_{4b}$ lie in
distinct orbits by Sylvester's Law of Inertia. On the other hand,
for $m<0$, the signatures match and, by using a similar argument to
the one used in the rank 3 case, that is applying $
\varphi(\tilde{C})\psi(D)\varphi(C) $ such that
$\tilde{C}^\sharp=D^\sharp=C^\sharp=0$, one can indeed verify they
are both indeed related to the canonical form $x_{4c}$ of the
theorem.
\qed\end{proof}

Note, the $\Fnt$ case considered here may be generalised to an
arbitrary  pseudo-Euclidean signature $\FTS^{p,q}:=\FTS(\J_{p-1,
q-1})$, where $\J_{p-1, q-1}=\R\oplus\Gamma_{p-1, q-1}$ was
introduced in \eqref{eq:n4gen}. The automorphism group is given by,
\be \Aut(\FTS^{p,q})=\SL(2, \R)\times\SO(p,q). \ee In particular,
$\mathcal{N}=4$ Maxwell-Einstein supergravity has an $\SL(2,
\R)\times\SO(6,q)$ U-duality and is related to
$\FTS^{6,q}:=\FTS(\J_{5, q-1})$. The analysis goes through
analogously  so we will not treat it in detail here. See, for
example, \cite{Krutelevich:2004,Gunaydin:2009zza,ICL-2} for further
details.

\subsubsection{Special Cases: $\FTS_{3\R}, \FTS_{2\R}$ and $\FTS_{\R}$}\label{sec:d4special}

\paragraph{Case 1: $\FTS(\J_{3\R})$} This is the $n=2$ point of the generic sequence $\Fnt$, as presented in \autoref{thm:redcanforms}. However, as mentioned in \autoref{sec:d5special}, the underlying Jordan algebra $\J_{1,1}=\R\oplus\Gamma_{1,1}$ may be reformulated in particularly symmetric manner as $\J_{3\R}=\R\oplus\R\oplus\R$, where $N(A)=a_1a_2a_3$ for $ (a_1, a_2, a_3)\in \J_{3\R}$. The permutation symmetry of the cubic norm is further manifested in the corresponding FTS, $\FTS^{2,2}\cong\FTS_{3\R}$. The elements of $\R\oplus\R\oplus\J_{3\R}\oplus\J_{3\R}$ may be written as a $2\times 2\times2$ \emph{hypermatrix}, denoted $a_{ABC}$:
\be
x=\begin{pmatrix}a_{000}&A=(a_{011}, a_{101}, a_{110})\\B=(a_{100}, a_{010},
a_{001})&a_{111}\end{pmatrix}\mapsto a_{ABC}, \quad \text{where} \quad A,B,C=0,1.
\ee
The permutation symmetry of the cubic norm corresponds to  its invariance under $A\leftrightarrow B\leftrightarrow C$. The hypermatrix lies in the fundamental representation $V_A\otimes V_B\otimes V_C$, where $V_i$ is a 2-dimensional real vector space,  of the automorphism group $\SL(2, \R)\times \SO(2,2)\cong \SL_A(2, \R)\times \SL_B(2, \R) \times \SL_C(2, \R)$. Explicitly,
\be
a_{ABC}\mapsto \tilde{a}_{ABC}=M_{A}{}^{A'}N_{B}{}^{B'}P_{C}{}^{C'}a_{A'B'C'},
\ee
where $M, N, P$ are $2\times 2$ matrices with determinant 1. The quartic norm is given by
Cayley's \emph{hyperdeterminant} $\Det a_{ABC}$ \cite{Cayley:1845},
\be \Delta(x)=-\Det
a=\frac{1}{2}\epsilon^{A_1A_2}\epsilon^{B_1B_2}\epsilon^{C_1C_3}\epsilon^{A_3A_4}\epsilon^{B_3B_4}\epsilon^{C_2C_4}a_{A_1B_1C_1}a_{A_2B_2C_2}a_{A_3B_3C_3}a_{A_4B_4C_4},
\ee
where $\epsilon$ is the antisymmetric $2\times 2$ invariant tensor of $\SL(2)$. This form of the quartic norm makes the $A\leftrightarrow B\leftrightarrow C$ triality invariance manifest.

\paragraph{Case 2: $\FTS(\J_{2\R})$} This is the $n=1$ point of the generic sequence $\Fnt$. However, since the underlying Jordan algebra $\J_{1,0}=\R\oplus\Gamma_{1,0}$ is Euclidean  the orbits of \autoref{thm:redcanforms} containing light-like $a_\mu\in\Gamma_{1,n-1}$ cannot be present. Indeed, we have the following result:
\begin{theorem}\label{thm:st2canform} Every element $x\in \FTS_{2\R}$ of a given rank is $\SL(2, \R)\times\SO(2, 1)$ related to one of the following canonical forms:
\begin{enumerate}
\item Rank 1
\begin{enumerate}
\item  $x_{1}=\begin{pmatrix}1&0\\0&0\end{pmatrix}$
\end{enumerate}
\item Rank 2
\begin{enumerate}
\item  $x_{2a}=\begin{pmatrix}1&(1;0)\\0&0\end{pmatrix}$
\item  $x_{2b}=\begin{pmatrix}1&(-1;0)\\0&0\end{pmatrix}$
\end{enumerate}
\item Rank 3
\begin{enumerate}
\item  $x_{3a}=\begin{pmatrix}1&(0;1)\\0&0\end{pmatrix}$
\end{enumerate}
\item Rank 4
\begin{enumerate}
\item   $x_{4a}=k\begin{pmatrix}1& (-1;1)\\0&0\end{pmatrix}$
\item   $x_{4b}=k\begin{pmatrix}1& (1;1)\\0&0\end{pmatrix},$
\end{enumerate}
where $k>0$.
\end{enumerate}
\end{theorem}
\begin{proof}
Since this is essentially a simplification of \autoref{thm:redcanforms} we will not present the details here. The key observation is that, since every rank 1, 2 and 3 element  of $\J_{1,0}$ is respectively of the form $(a;0), (0; a_0)$ and $(a; a_0)$, where $a, a_0\not=0$, the  $x_{2c}, x_{3b}$ and $x_{4c}$ orbits of \autoref{thm:redcanforms} are absent.
\qed\end{proof}
The underlying Jordan algebra $\J_{1,0}=\R\oplus\Gamma_{1,0}$ may be
written as a degeneration $\J_{3\R}\rightarrow\J_{2\R}=\R\oplus\R$.
At the level of the  FTS $\FTS_{3\R}\rightarrow\FTS_{2\R}$, this
corresponds to symmetrizing the  $2\times 2\times2$ hypermatrix of
$\FTS_{3\R}$ over two  indices: $a_{ABC}\rightarrow a_{A(B_1B_2)}$.
The  partially symmetrized  hypermatrix lies in the   $V_A\otimes
\text{Sym}^2 (V_B)$ representation  of the automorphism group
$\SL(2, \R)\times \SO(2,1)\cong \SL_A(2, \R)\times \SL_B(2, \R)$.
Explicitly, \be a_{A(B_1B_2)}\mapsto
\tilde{a}_{A(B_1B_2)}=M_{A}{}^{A'}N_{B_1}{}^{B'_1}N_{B_2}{}^{B'_2}a_{A'(B'_1B'_2)},
\ee where $M, N$ are $2\times 2$ matrices with determinant 1. The
quartic norm is again given by the hyperdeterminant through applying
the appropriate symmetrization to $\Det a_{ABC}$. For more details
see, for example,
\cite{Bhargava:2004,Krutelevich:2004,Bellucci:2007zi}.

\paragraph{Case 3: $\FTS(\J_{\R})$}  May be regarded as the end point of this sequence, in the sense that $\FTS_{\R}=\R\oplus\R\oplus\J_{\R}\oplus\J_{\R}$  can  be mapped to the space of totally symmetrized hypermatrices: $a_{(A_1A_2A_3)}\in \text{Sym}^3(V_A)$. The  totally symmetrized  hypermatrix transforms as
\be
a_{(A_1A_2A_3)}\mapsto \tilde{a}_{(A_1A_2A_3)}=M_{A_1}{}^{A'_1}M_{A_2}{}^{A'_2}M_{A_3}{}^{A'_3}a_{(A'_1A'_2A'_3)},
\ee
where $M$ is a $2\times 2$ determinant 1 matrix, under the automorphism group $\SL_A(2, \R)$. Once again the quartic norm is  given by
the hyperdeterminant by totally ``symmetrizing'' $\Det a_{ABC}$. For more details see, for example, \cite{Bhargava:2004,Krutelevich:2004,Bellucci:2007zi}.

As already noted in \autoref{sec:d5special}, because $N(A)=A^3, A^\sharp=A^2$, all non-zero elements  $A\in\J_\R$ are rank 3. Consequently the number of independent ranks in $\FTS_\R$ is reduced to three:

\begin{lemma} If $x\in \FTS_{\R}$ is rank 2, then it is rank 1.
\end{lemma}
\begin{proof}
Consider the independent  rank 2 conditions evaluated on the reduced
form of \eqref{eq:redform}: \be \textrm{Rank}\; x_{\text{red}}
\leq2\Leftrightarrow A^\sharp=0,\;\beta=0. \ee Since
$A^\sharp=0\Rightarrow A=0$ for all $A\in \J_\R$, one obtains \be
\textrm{Rank}\; x_{\text{red}} \leq2\Rightarrow
A=0,\;\beta=0\Rightarrow \textrm{Rank}\; x_{\text{red}} =1. \ee
\qed\end{proof}
Hence, the rank 2 orbits of  \autoref{thm:redcanforms} do not exist
for $\FTS_\R$. There are only elements of rank 1, 3 or 4, and we
have the following
\begin{theorem}\label{thm:T4} Every element $x\in \FTS_{\R}$ of a given rank is $\SL(2, \R)$ related to one of the following canonical forms:
\begin{enumerate}
\item Rank 1
\begin{enumerate}
\item  $x_{1}=\begin{pmatrix}1&0\\0&0\end{pmatrix}$
\end{enumerate}
\item Rank 3
\begin{enumerate}
\item  $x_{3}=\begin{pmatrix}0&1\\0&0\end{pmatrix}$
\end{enumerate}
\item Rank 4
\begin{enumerate}
\item   $x_{4a}=k\begin{pmatrix}1& -1\\0&0\end{pmatrix}$
\item   $x_{4b}=k\begin{pmatrix}1& 1\\0&0\end{pmatrix},$
\end{enumerate}
where $k>0$.
\end{enumerate}
\end{theorem}
\begin{proof}
We begin by transforming to the generic canonical form  \eqref{eq:redform}
 and proceed, case by case, according to the rank.
\paragraph{Rank 1:}
$ \text{Rank} x=1\Rightarrow A=0,\beta=0, $ so that every rank 1
element is $\Aut(\FTS_{\R})$ related to \be
x_1=\begin{pmatrix}1&0\\0&0\end{pmatrix}. \ee
\paragraph{Rank 3:}
$\text{Rank} x=3\Rightarrow 4N(A)=4A^3=-\beta^2$ and $A\not=0$, so
that every rank 3 element is $\Aut(\FTS_{\R})$ related to a reduced
form \be\label{eq:r3half}
\begin{pmatrix}1&A\\0&\sqrt{-A^3}\end{pmatrix},
\ee with $A<0$. In order to determine the $\Aut(\FTS_{\R})$
transformation bringing \eqref{eq:r3half} into the desired form, it
is convenient to use the totally symmetric hypermatrix
representation of the charges: \be x=a_{(A_1A_2A_3)}, \quad A_1,
A_2, A_3=0,1 \ee where, explicitly, \be
\begin{split}
a_{(000)}=\alpha, &\quad a_{(110)}=a_{(101)}=a_{(011)}=A;\\
a_{(000)}=\beta, &\quad a_{(001)}=a_{(010)}=a_{(100)}=B.\\
\end{split}
\ee A generic $\Aut(\FTS_{\R})$ transformation is then given by an
$\SL(2, \R)$ matrix \be M=\begin{pmatrix}a&b\\c&d\end{pmatrix},
\quad ad-bc=1 \ee under which, \be a_{(ABC)}\mapsto
M_{A}{}^{A'}M_{B}{}^{B'}M_{C}{}^{C'}a_{(A'B'C')}=\tilde{a}_{(ABC)}.
\ee Applying $M$ to the reduced form \eqref{eq:r3half}, we see that
in order to obtain $x_3$ we are required to solve the follow system
of polynomial equations

\begin{align}
\tilde{a}_{(000)}&=a^3-3|A|b^2 a+2|A|^{\frac{3}{2}}b^3=0;\label{eq:r3a}\\
\tilde{a}_{(110)}&=c^2a-|A|(ad^2 +2bcd)+2|A|^{\frac{3}{2}}d^2b=\tilde{A};\label{eq:r3b}\\
\tilde{a}_{(001)}&=a^2c-|A|(cb^2 +2dab)+2|A|^{\frac{3}{2}}b^2d=0;\label{eq:r3c}\\
\tilde{a}_{(111)}&=c^3-3|A|d^2 c+2|A|^{\frac{3}{2}}d^3=0,\label{eq:r3d}
\end{align}
where we leave $\tilde{A}\not=0$ arbitrary as it may be subsequently
scaled away using \be
M=\begin{pmatrix}\tilde{A}&0\\0&\tilde{A}^{-1}\end{pmatrix}. \ee
Setting $d=1$, one immediately sees that \eqref{eq:r3d} has two
distinct real roots: \be
(c+2|A|^{\frac{1}{2}})(c-|A|^{\frac{1}{2}})^2=0. \ee The double root
$c=|A|^{\frac{1}{2}}$ contradicts $\tilde{A}\not=0$ so we choose
$c=-2|A|^{\frac{1}{2}}$, which implies $a+2|A|^{\frac{1}{2}}b=1$
and, from \eqref{eq:r3b}: \be \tilde{A}=3|A|. \ee Substituting $d=1,
c=-2|A|^{\frac{1}{2}}$ into \eqref{eq:r3c} and solving for $a$ we
find, \be a_{\pm}=\frac{|A|^{\frac{1}{2}} b}{2}(-1\pm 3). \ee
Letting $a=a_+$ we determine that \be
M=\begin{pmatrix}\frac{1}{3}&\frac{1}{3|A|^{\frac{1}{2}}}\\-2|A|^{\frac{1}{2}}&1\end{pmatrix}
\ee sends \eqref{eq:r3half} to $
\bigl(\begin{smallmatrix}0&3|A|\\0&0\end{smallmatrix} \bigr), $
which is related by $
M=\bigl(\begin{smallmatrix}3|A|&0\\0&(3|A|)^{-1}\end{smallmatrix}
\bigr) $
 to
$
\bigl(\begin{smallmatrix}0&1\\0&0\end{smallmatrix} \bigr),
$
as required.
\paragraph{Rank 4:}
$\text{Rank} x=4\Rightarrow 4A^3+\beta^2\not=0$. First, we show that
starting from \eqref{eq:genred} every rank 4 $x$ may be brought into
a form with $B=\beta=0$, namely: \be\label{eq:aA}
\begin{pmatrix}\tilde{\alpha}&\tilde{A}\\0&0\end{pmatrix}.
\ee This amounts to solving
\begin{align}
\tilde{a}_{(001)}&=a^2c+A(cb^2 +2dab)+\beta b^2d=0;\label{eq:r4c}\\
\tilde{a}_{(111)}&=c^3+3Ad^2 c+\beta d^3=0,\label{eq:r4d}
\end{align}
where $ad-bc=1$. There are three subcases: (i) $A\not=0$, $\beta=0$,
(ii) $A=0$, $\beta\not=0$, (iii) $A\not=0$, $\beta\not=0$. (i) is
trivial. For (ii), our system simplifies down to
\begin{align}
\tilde{a}_{(001)}&=a^2c+\beta b^2d=0;\label{eq:r4c2}\\
\tilde{a}_{(111)}&=c^3+\beta d^3=0,\label{eq:r4d2}
\end{align}
Setting $d=1$ and $c=-\beta^{\frac{1}{3}}$ solves  \eqref{eq:r4d2}.
By substituting this choice into \eqref{eq:r4c2} and solving for
$a$, one finds $a_{\pm}=\pm\beta^{\frac{1}{3}}b$. But only  $a_{+}$
is consistent with $ad-bc=1$. Making this choice implies
$b=(8\beta)^{-\frac{1}{3}}$, and one obtains the $\SL(2, \R)$ matrix
 \be
 M=\begin{pmatrix}\frac{1}{2}&-\beta^{\frac{1}{3}}\\(8\beta)^{-\frac{1}{3}}&1\end{pmatrix}.
 \ee

Finally, let us consider the case (iii) $A\not=0$, $\beta\not=0$. Let $c=\gamma d$, where
$\gamma=\gamma(\beta, A)$. Then, from \eqref{eq:r4d}, we have \be
d^3(\gamma^3+3A\gamma+\beta)=0, \ee which, since $d$ is necessarily
non-zero, implies \be\label{eq:cubiceq}
f(\gamma)=\gamma^3+3A\gamma+\beta=0. \ee There is at least one real
root $\gamma_*$ that is non-zero for $\beta\not=0$. Substituting
into \eqref{eq:r4c} yields, \be \gamma_*
d\left[a^2+\frac{2Ab}{\gamma_*}a+(A+\frac{\beta}{\gamma_*})b^2\right]=0.
\ee Solving for $a$ we find \be
a_{\pm}=\frac{Ab}{\gamma_*}\left(-1\pm\sqrt{1-\frac{\gamma_{*}^{2}}{A^2}(A+\frac{\beta}{\gamma_*})}\right)=\xi_\pm(\beta,
A)b. \ee Hence, we require \be\label{eq:areq}
A^2-\gamma_{*}^{2}A-\gamma_*\beta=A^2-y(\gamma_*)\geq 0, \ee where $y(\gamma_*)=\gamma_{*}^{2}A+\gamma_*\beta$. We may always choose
the root $\gamma_*$ such that this condition holds. In order to see
this, let us consider the two subcases: $(a)$ $A<0$ and $(b)$ $A>0$.
$(a)$ For $A<0$, $f(\gamma)$ in \eqref{eq:cubiceq}
 has two turning points at $\pm\sqrt{|A|}$. Consequently, for $\beta>0$ there is always a real root $\gamma_*< -\sqrt{|A|}<0$, which implies \eqref{eq:areq}. Similarly, for $\beta<0$ there is always a real root $\gamma_*> \sqrt{|A|}>0$, which again implies \eqref{eq:areq}.  $(b)$ For $A>0$ the cubic $f(\gamma)$ only has  an inflection point at $\gamma=0, f(0)=\beta$ and so $f(\gamma)=0$ has a single real root $\gamma_*$. If $\beta<0$, then $0<\gamma_*<\frac{|\beta|}{3A}$. Since $y(\gamma_*)$ has a minimum at $\frac{-\beta}{2A}$ such that $y(\frac{-\beta}{2A})=\frac{-\beta^2}{4A}<0$, it is clear that $0<\gamma_*<\frac{|\beta|}{3A}$ implies $y(\gamma_*)<0$ and so condition \eqref{eq:areq} is satisfied. Similarly, if $\beta>0$, then
$-\frac{\beta}{3A}<\gamma_*<0$, and once again $y(\gamma_*)<0$, implying condition \eqref{eq:areq} as required.
Hence, we conclude condition \eqref{eq:areq} may always be satisfied.

In summary: \be a=\xi_\pm(\beta, A)b,\quad c=\gamma_*(\beta, A)d \ee
which yields \be (\xi_\pm-\gamma_*)bd=1, \ee where $\xi_+$ or
$\xi_-$ is chosen such that $\xi_\pm-\gamma_*$ is
non-zero\footnote{Note that $\xi_\pm=\gamma_*$ for both choices
implies $\gamma_*=\pm\sqrt{-A}$ which, from \eqref{eq:cubiceq},
implies $4A^3+\beta^2=0$, in turn contradicting our rank 4
assumption.}. Without loss of generality we can set $d=1$, so
that $b=1/(\xi_\pm-\gamma_*)$ and the $\SL(2, \R)$ matrix
 \be
 M=\begin{pmatrix}\xi_\pm/(\xi_\pm-\gamma_*)&1/(\xi_\pm-\gamma_*)\\\gamma_*&1\end{pmatrix},
 \ee
transforms our hypermatrix into the desired form \eqref{eq:aA}.

Finally, the reduced form \eqref{eq:aA} may be brought into the form
$x_{4a/b}$ of the theorem by the diagonal $\SL(2, \R)$
transformation
 \be
 M=\begin{pmatrix}\lambda &0\\0&\lambda^{-1}\end{pmatrix}:\begin{pmatrix}\tilde{\alpha}&\tilde{A}\\0&0\end{pmatrix}\mapsto k\begin{pmatrix}1&\epsilon\\0&0\end{pmatrix}.
 \ee
 where $\epsilon=+1,-1$ according as $\Delta>0, \Delta<0$.
\qed\end{proof}

\section*{Acknowledgments}

We would like to thank Duminda Dahanayake for useful discussions. The work of LB and SF is supported by the ERC Advanced Grant no. 226455 \textit{%
SUPERFIELDS}. Furthermore, the work of SF is also supported in
part by DOE Grant DE-FG03-91ER40662. The work of MJD is
supported by the STFC under rolling grant ST/G000743/1. LB is grateful for hospitality at the Theoretical Physics group at Imperial College London  and the CERN theory division  (where he was supported by the
above ERC Advanced Grant).


\providecommand{\href}[2]{#2}\begingroup\raggedright\endgroup

\end{document}